\documentclass[11pt]{amsart}
\usepackage{amssymb,amsmath,amsfonts,amscd,euscript}
\usepackage[all]{xy}
\usepackage{color}

\usepackage[backref=page]{hyperref}
\usepackage{todonotes}

\usepackage{amsthm,enumitem,bm,pict2e,xcolor}
\usepackage{indentfirst}

\usepackage{mathtools, bbm}

\newcommand{\nc}{\newcommand}

\numberwithin{equation}{section}
\newtheorem{thm}{Theorem}[section]
\newtheorem{prop}[thm]{Proposition}
\newtheorem{lem}[thm]{Lemma}
\newtheorem{cor}[thm]{Corollary}
\newtheorem{rem}[thm]{Remark}
\newtheorem{definition}[thm]{Definition}
\newtheorem{example}[thm]{Example}
\newtheorem{dfn}[thm]{Definition}

\newtheorem{constr}[thm]{Construction}

\nc{\Pc}{\bar{P}}
\nc{\cC}{\mathcal{C}}
\nc{\one}{\mathbbm{1}}
\newtheorem{claim}{Claim}[thm]

\nc{\gl}{\mathfrak{gl}}
\nc{\GL}{\mathfrak{GL}}
\nc{\g}{\mathfrak{g}}
\nc{\gh}{\widehat\g}
\nc{\h}{\mathfrak{h}}
\nc{\la}{\lambda}
\nc{\al}{\alpha }
\nc{\be}{\beta }
\nc{\ve}{\varepsilon }
\nc{\om}{\omega }

\nc{\ta}{\theta}
\nc{\veps}{\varepsilon}
\nc{\ch}{{\mathop {\rm ch}}}
\nc{\Tr}{{\mathop {\rm Tr}\,}}
\nc{\Id}{{\mathop {\rm Id}}}
\nc{\bra}{\langle}
\nc{\ket}{\rangle}
\nc{\x}{{\bf x}}
\nc{\bs}{{\bf s}}
\nc{\bp}{{\bf p}}
\nc{\bt}{{\bf t}}
\nc{\pa}{\partial}
\nc{\ld}{\ldots}
\nc{\cd}{\cdots}
\nc{\hk}{\hookrightarrow}
\nc{\T}{\otimes}
\newcommand{\bea}{\begin{equation}}
\newcommand{\ena}{\end{equation}}
\nc{\gr}{\mathrm{gr}}
\nc{\Gr}{\mathrm{Gr}}
\nc{\ov}{\overline}

\nc{\cO}{\mathcal O}
\nc{\cF}{\mathcal F}
\nc{\cL}{\mathcal L}
\nc{\msl}{\mathfrak{sl}}
\nc{\mgl}{\mathfrak{gl}}
\nc{\U}{\mathrm U}
\nc{\V}{\EuScript V}
\nc{\bH}{\EuScript H}
\nc{\Res}{\mathrm{Res\ }}

\newcommand{\bC}{{\mathbb C}}
\newcommand{\bR}{{\mathbb R}}
\newcommand{\bZ}{{\mathbb Z}}

\newcommand{\bP}{{\mathbb P}}
\newcommand{\bA}{{\mathbb A}}
\newcommand{\bG}{{\mathbb G}}
\newcommand{\fp}{{\mathfrak p}}

\newcommand{\G}{{\mathfrak G}}
\newcommand{\I}{{\mathfrak I}}

\newcommand{\fn}{{\mathfrak n}}
\newcommand{\fa}{{\mathfrak a}}

\begin{document}

\title[Birational maps, representations  and poset polytopes]
{Birational maps to Grassmannians, representations  and poset polytopes}

\author{Evgeny Feigin}
\address{Evgeny Feigin:\newline
School of Mathematical Sciences, Tel Aviv University, Tel Aviv, 69978, Israel}
\email{evgfeig@gmail.com}

\makeatletter\let\@wraptoccontribs\wraptoccontribs\makeatother
\contrib[an appendix in collaboration with]{Wojciech Samotij}
\email{samotij@tauex.tau.ac.il}

\dedicatory{To Peter Littelmann with heartfelt gratitude}

\begin{abstract}
We study the closure of the graph of the birational map from a projective space to a Grassmannian.
We provide explicit description of the graph closure and compute the fibers of 
the natural projection to the Grassmannian. We construct embeddings of
the graph closure to the projectivizations of certain cyclic representations of a degenerate special linear Lie algebra 
and study algebraic and combinatorial properties of these representations.
In particular, we describe monomial bases, generalizing the FFLV
bases. The proof relies on combinatorial properties of a new family of poset polytopes, which are of independent interest.
As a consequence we obtain flat toric degenerations of 
the graph closure studied by Borovik, Sturmfels and Sverrisd\'ottir.    
\end{abstract}

\maketitle

\section*{Introduction}
Let $\Gr(d,n)$ be the Grassmannian of $d$-dimensional subspaces in an $n$-dimensional vector space. The Grassmann variety admits a cellular decomposition into Bruhat cells with a unique open cell $\bA^N$, where $N=d(n-d)$
is the dimension of $\Gr(d,n)$. 
Hence one gets  the birational exponential map $\imath: \bP^N\to \Gr(d,n)$. 
Our main object of study
is the closure of the graph of this map, i.e.\ the variety $\G(d,n)\subset \bP^N\times \Gr(d,n)$, which is the closure of the set $(x,\imath(x))$ for $x\in\bA^N$. 
Some results about these varieties were obtained in \cite{BSS,FSS} (see also \cite{KP}
for the symplectic case).

By definition, there is a surjective map $\G(d,n)\to \Gr(d,n)$. Our first task is to describe the fibers of this map. 
To this end, let us denote by $L$ the $n$-dimensional vector space; 
any $U\in \Gr(d,n)$ is a $d$-dimensional subspace of $L$.
We fix a decomposition $L=L^-\oplus L^+$, where $\dim L^-=d$ and $\dim L^+=n-d$.
Consider the stratification
$\Gr(d,n)=\bigsqcup_{k=0}^{\min(d,n-d)} X_k$, where 
$X_k$ consists of subspaces $U\subset L$ such that $\dim (U\cap L^+) = k$
and $X_0$ is the image of the map $\imath: \bA^N\to \Gr(d,n)$. We prove the following theorem.

\begin{thm}
The map $\varphi$ is one-to-one over $X_0\sqcup X_1$. For 	
$k\ge 2$ the map $\varphi$ over  $X_k$ is a fiber bundle with
a fiber isomorphic to the projective space $\bP^{k^2-1}$. 
\end{thm}

In order to study the varieties $\G(d,n)$ we consider the action of a certain
group $SL_n^{(d)}$ on $\G(d,n)$. This group is a degeneration of the classical group 
$SL_n$ depending on the parameter $d$, it contains
a normal  abelian subgroup $N_d$ (isomorphic to an abelian unipotent
radical of $SL_n$) and the maximal standard parabolic subgroup $P_d$ of $SL_n$
(corresponding to the $d$-th simple root). More precisely, 
$SL_n^{(d)}\simeq N_d\ltimes P_d$.
We denote the Lie algebra of $SL_n^{(d)}$
by $\msl_n^{(d)}$ and the (abelian) Lie algebra of $N_d$ by $\fa_d$.

\begin{constr}
For each pair $m,M\in\bZ_{>0}$  there exists an $\msl_n^{(d)}$ module
$L_{m,M}$, 
which is cyclic with respect to the action of $\fa_d$ such that $\G(d,n)$
embeds into $\bP(L_{m,M})$ as the closure of the $N_d$ orbit through the cyclic line. 
\end{constr}

In particular,  $\G(d,n)$ are the $\bG_a^N$ varieties 
(see \cite{Ar,AZ,F1,HT}).
We note that the modules $L_{m,M}$ are defined for all pairs $m,M\in\bZ_{\ge 0}$,
and for $m=0$ or $M=0$ the map $\G(d,n)\to \bP(L_{m,M})$ still exists, but 
is not an embedding. The spaces $L_{m,M}$ are responsible
for the description of the homogeneous coordinate rings of $\G(d,n)$ with respect
to the generalized Pl\"ucker embeddings and for the spaces of sections of natural
line bundles, so  we are interested in algebraic and combinatorial properties of 
$L_{m,M}$. More precisely, we want to describe $L_{m,M}$ as $\msl_n^{(d)}$ modules
(via explicit relations), find monomial bases and compute the dimensions.
     
To this end, we generalize the construction of the FFLV bases (see \cite{FFL1}). Namely,
we consider the set of roots $P$ of $\fa_d$ and for $\al\in P$ we denote by
$f_\al\in\fa_d$ the corresponding root vector. The set $P$ consists of roots 
$\al_{i,j}$ with $1\le i\le d\le j\le n-1$ and has a natural structure of 
a poset. In particular, the Dyck paths from \cite{FFL1}  are chains in $P$.
We define the set $S(m,M)\subset \bZ_{\ge 0}^P$ consisting of collections 
$(s_\al)_{\al\in P}$ subject to the conditions: 
\begin{equation}\label{P'}
\sum_{\al\in P'} s_\al\le m\cdot w(P') + M \text{ for all } P'\subset P,
\end{equation}
 where $w(P')$ is the width of the subposet $P'$ (the minimal number of chains 
 needed to cover $P'$).   We prove the following theorem.
 \begin{thm}
 The elements $\prod_{\al\in P} f_\al^{s_\al}$, $(s_\al)_\al\in S(m,M)$ being applied
 to the cyclic vector form a basis of $L_{m,M}$. 
 \end{thm}

To prove this theorem we consider the polytopes $X(m,M)\subset \bR_{\ge 0}^P$ defined by inequalities \eqref{P'} (for an arbitrary finite poset $P$).
In particular, $S(m,M)=X(m,M)\cap \bZ_{\ge 0}^P$.
The poset polytopes proved to be useful in various problems (see e.g. \cite{FaFo, FM,M2,Stur,Stan}).
We prove the following theorem.

\begin{thm}\label{Minsum}
Let $P$ be a finite poset. Then $X(m,M)$ are lattice polytopes such that
$X(m_1,M_1) + X(m_2,M_2) = X(m_1+m_2,M_1+M_2)$ and 
$S(m_1,M_1) + S(m_2,M_2) = S(m_1+m_2,M_1+M_2)$ for arbitrary  $m_1,m_2,M_1,M_2\ge 0$.
\end{thm}

In particular, the polytopes $X(m,M)$ have the integer decomposition property (see \cite{HOT,T}).  
As an application of Theorem \ref{Minsum} we obtain that the polytopes $X(m,M)$ define flat toric degenerations of $\G(d,n)$ (see \cite{An,AB,FaFL1,O}). 
Our approach is in some sense dual to the one used in \cite{BSS},
where the Pl\"ucker type relations for $\G(d,n)$ are utilized.  
We note that in \cite{BSS} the authors consider the Cayley sums, which are closely related to the Minkowski sums (see e.g. \cite{T}).

Let us close with the following remark. There are two natural directions we plan 
to address elsewhere. First, one can replace the Grassmannians with arbitrary 
flag varieties  of type $A$ or their PBW degenerate versions  
(see \cite{ABS,FaFL2,F1,F2,FeFi} for the corresponding results
on the PBW degenerations). Second, it is interesting to consider symplectic and orthogonal
cases where the PBW structures are also available (see \cite{BK,FFLi,FFL2,G,M1}).

The paper is organized as follows. 
In Section \ref{combinatorics} we introduce the polytopes $X(m,M)$ and study their properties. 
The proof of the main theorem is given in Appendix \ref{sec:appendix}.
In Section \ref{algebra}  we recall the construction of the FFLV bases
in the irreducible highest weight representations. We also introduce
the modules $L_{m,M}$ and study their algebraic properties using the polytopes $X(m,M)$.
In Section \ref{geometry} we introduce the graph closures  $\G(d,n)$ and 
describe the fibers of the natural projection map $\G(d,n)\to \Gr(d,n)$.
We use the results from the combinatorial and algebraic sections to describe
the geometric properties of $\G(d,n)$.

\section*{Acknowledgments}
We are grateful to Alexander Kuznetsov and  Igor Makhlin for useful discussions and correspondence.

\section{Combinatorics}\label{combinatorics}
In this section we attach a family of polytopes $X(m,M)$, $m, M\in \bZ_{\ge 0}$ 
to a finite poset $P$ and show that these polytopes satisfy Minkowsky property with respect to the parameters.
The polytopes $X(m,M)$ play a crucial role in the next sections.

Let $P$ be a finite poset. For a subposet $P'\subset P$ we denote by $w(P')$ the width of $P'$, which can be defined as a minimal number 
of chains needed to cover $P'$ or, equivalently, as the maximal length of an antichain in $P'$ (Dilworth's theorem \cite{D}). 
For  two nonegative integers $m$ and $M$ we define a polytope $X(m,M)\subset \bZ_{\ge 0}^P$
consisting of points $(x_\al)_{\al\in{P}}$ subject to the conditions 
\begin{equation}\label{ineq}
\sum_{\al\in P'} x_\al\le m\cdot w(P') + M \text{ for all } P'\subset P.
\end{equation}
We denote by $S(m,M)$ the set of integer points in $X(m,M)$.

\begin{example}
Let $m=0$. Then $X(0,M)$ is a scaled simplex of $P$ (the sum of all coordinates does 
not exceed $M$).  
\end{example}

\begin{example}\label{M=0}
Let $M=0$. Then $X(m,0)$ is a scaled chain polytope of $P$ (see \cite{Stan}).
In fact, the defining inequalities are $\sum_{\al\in P'} x_\al\le m\cdot w(P')$ for all $P'\subset P$,
which are implied by the subset of inequalities with $w(P')=1$ (i.e.\ for $P'$ being a single chain).
Therefore, $X(m,0)$ consists of elements $(x_\al)_{\al\in P}$ such that the 
sum over any chain does not 
exceed $m$. Hence $X(m,0)=m\Delta_{chain}(P)$, where $\Delta_{chain}(P)$ is the Stanley chain polytope.
In particular, the points of $S(1,0)$ are the indicator functions of antichains in $P$ and 
 $S(m,0)$ are all possible sums of $m$ indicator functions.
\end{example}

\begin{thm}\label{Msum}
For any $m_1,m_2,M_1,M_2\ge 0$ one has 
\[
S(m_1,M_1)+S(m_2,M_2) = S(m_1+m_2,M_1+M_2).
\]
\end{thm}
\begin{proof}
We note that $S(m_1,0)+S(m_2,0) = S(m_1+m_2,0)$ (see Example \ref{M=0}). Hence it suffices to show that 
$S(m,M)+S(0,1) = S(m,M+1)$. This is shown in Theorem \ref{M+1}.
\end{proof}

\begin{cor}\label{Xcor}
The polytopes $X(m,M)$ are normal lattice polytopes. 
For any nonnegative $m_1,m_2,M_1,M_2$ one has 
\[
X(m_1,M_1)+X(m_2,M_2) = X(m_1+m_2,M_1+M_2).
\]
\end{cor}
\begin{proof}
Assume that $p$ is a rational vertex of $X(m,M)$. Then there exists $r>0$ such that $rp$ has integer coordinates and hence 
$rp\in S(rm,rM)$. Now Theorem \ref{Msum}  implies that $rp$ is equal to a sum of $r$ points in $S(m,M)$: $rp=p_1+\dots + p_r$, $p_i\in S(m,M)$. Since $rp$ is a vertex
of $X(rm,rM)$, we conclude that $p_i=p$ and hence $p\in S(m,M)$. So $X(m,M)$ are lattice polytopes. 

The claim $X(m_1,M_1)+X(m_2,M_2) = X(m_1+m_2,M_1+M_2)$ is proved in a similar way.
First, let $p\in X(m_1+m_2,M_1+M_2)$ be a rational point. Then there exists an $r>0$
such that $rp\in S(r(m_1+m_2),r(M_1+M_2))$. Hence there exists $y_1\in S(rm_1,rM_1)$,
$y_2\in S(rm_2,rM_2)$ such that $rp=y_1+y_2$. Now $y_1/r\in X(m_1,M_1)$ and
$y_2/r\in X(m_2,M_2)$ and hence $p\in X(m_1,M_1)+X(m_2,M_2)$. Since all polytpoes 
$X(m,M)$ are lattice, we conclude that the desired Minkowski sum property holds 
in general.
\end{proof}

Let us derive a corollary from Theorem \ref{Msum}.
\begin{cor}\label{corP}
Let $P$ be a finite poset such that for some $m,M\in\bZ_{\ge 0}$ one has 
\[
|P'| \le m\cdot w(P') + M \text{ for all } P'\subset P.
\]
Then there exists a decomposition of $P$ into the union of disjoint sets
\[
P = A_1 \sqcup \ldots \sqcup  A_m \sqcup B,
\]   
where $A_i$ are antichains and $|B|\le M$.
\end{cor} 
\begin{proof}
To show that Theorem \ref{Msum} implies Corollary \ref{corP} it suffices to take a point 
$\bs\in \bZ_{\ge 0}^P$ such that $s_\al=1$ for all $\al\in P$. By definition, $\bs\in S(m,M)$ and
hence there exist $\bs^1,\dots,\bs^m\in S(1,0)$ such that $\bs-\sum_{i=1}^m \bs^i\in S(0,M)$.
Now $A_i$ is the support of $\bs^i$ and $B$ is the support of 
$\bs-\sum_{i=1}^m \bs^i\in S(0,M)$.    
\end{proof}

\section{Algebra}\label{algebra}
In this section we introduce and study  main algebraic objects we are interested in.
In particular, the polytopes from the previous section are used to construct
monomial bases in certain cyclic representations of abelian algebras.

\subsection{The PBW setup}
We fix nonnegative integers $n$ and $d$ such that $1\le d<n$. Let $[n]=\{1,\dots,n\}$.
Let $\fn_-\subset \msl_n$ be the Lie algebra of lower triangular matrices.
It contains the $d(n-d)$ dimensional abelian radical $\fa_d\subset \fn_-$ spanned by the matrix
units $E_{i,j}$, $i=d+1,\dots,n$, $j=1,\dots,d$. The same basis of $\fa_d$
can be described in terms of root vectors. Namely,   
let $\al_i$, $i\in [n-1]$ be simple roots of the Lie algebra $\msl_n$.
For $1\le i\le j\le n-1$ let $\al_{i,j}=\al_i+\dots +\al_j$ be the positive
roots. Let $f_{i,j}\in \fn_-$ be the root vector for the root $-\al_{i,j}$.
In particular, $f_{i,j}= E_{j+1,i}$. Then $\fa_d$ is spanned by $f_{i,j}$ with
$i\le d\le j$.  
 
Let $\omega_i$, $i\in [n-1]$ be the fundamental weights. For a dominant integral weight $\la=\sum m_i\omega_i$ we denote by $L_\la$ the
corresponding irreducible highest weight $\msl_n$ module with a highest weight
vector $\ell_\la\in L_\la$. 
In this paper we deal with $\la=m\omega_d$. In particular, $L_{\omega_d}\simeq \Lambda^d(L)$, where $L=L_{\omega_1}$ is the $n$-dimensional vector representation.
We fix a basis $\ell_1,\dots,\ell_n$ of $L$; for $I=\{i_1<\dots<i_d\}\subset [n]$    
let $\ell_I\in L_{\omega_d}$ be the wedge product $\ell_{i_1}\wedge\dots \wedge \ell_{i_d}$.

We note that the vector $\ell_{\om_d}=\ell_1\wedge\dots\wedge \ell_d$ is the highest weight vector 
of $L_{\omega_d}$. In particular, 
\[
L_{\omega_d}=\U(\fn_-)\ell_{\om_d}=\U(\fa_d)\ell_{\om_d}.
\] 
For an arbitrary $m>0$ the irreducible module $L_{m\omega_d}$ sits inside the 
tensor power $L_{\omega_d}^{\T m}$ as a Cartan component. More precisely,
\[
L_{m\omega_d} =  \U(\fa_d) \ell_{\omega_d}^{\T m}\subset L_{\omega_d}^{\otimes m},
\]
the highest weight vector $\ell_{m\om_d}\in L_{m\om_d}$ is identified with
$\ell_{\omega_d}^{\T m}$.
In general, for two cyclic $\fa_d$ modules $U_1=\U(\fa_d)u_1$ and 
$U_2=\U(\fa_d)u_2$ we define the cyclic $\fa_d$ module $U_1\odot U_2$
as 
\[
U_1\odot U_2 = \U(\fa_d) (u_1\T u_2)\subset U_1\T U_2.
\]
In particular, $L_{m_1\omega_d}\odot L_{m_2\omega_d}\simeq L_{(m_1+m_2)\omega_d}$.

Since $\fa_d$ is abelian, each representation $L_{m\omega_d}$ is naturally
identified with a quotient of the polynomial ring $\bC[f_{i,j}]$, 
$i\le d\le j$ by a certain ideal. We note that this ideal is graded with respect 
to the total degree of polynomials in  $\bC[f_{i,j}]$ (the degree
of each variable $f_{i,j}$ is one). Hence we get the induced PBW grading
on   $L_{m\omega_d}$:
\[
L_{m\omega_d}=\bigoplus L_{m\omega_d}(r),\ 
L_{m\omega_d}(r)=\mathrm{span} \{f_{i_1,j_1}\dots f_{i_r,j_r}\ell_{m\om_d}, i_\bullet\le d\le j_\bullet\}. 
\]

\begin{example}
If $m=1$, then $L_{\omega_d}(r)\subset \Lambda^d(L)$ is spanned by vectors 
$\ell_I$ such that $\# (I\cap [d]) = d-r$. In other words, the degree of a
wedge monomial $\ell_I$ is equal to the number of elements of $I$ outside 
$[d]$ (this the PBW degree in the terminology of \cite{FFL1} or the level in the terminology of \cite{FO,FSS}). 
\end{example}

In \cite{FFL1} (see also \cite{Vin}) the authors construct monomial bases
for the finite-dimensional irreducible highest weight representations of 
$\msl_n$. We recall the construction below (for highest weights $m\om_d$).
A Dyck path ${\bf p}=(p_1,\dots,p_{n-1})$ is a subset of the set $\al_{i,j}$, $i\le d\le j$
satisfying the following conditions: 
\begin{itemize}
	\item $p_1=\al_{1,d}$, $p_{n-1} = \al_{d,n-1}$,
	\item if $p_s=\al_{i_s,j_s}$, then $p_{s+1}=\al_{i_s+1,j_s}$ or $p_{s+1}=\al_{i_s,j_s+1}$. 
\end{itemize}     
\begin{rem}
Recall that the roots of $\fa_d$ are naturally identified with  
matrix units: $f_{i,j}= E_{j+1,i}$. Via this
identification, a Dyck path starts at $E_{d+1,1}$, ends at $E_{n,d}$ and 
at each step goes either one cell down or one cell right.   
\end{rem}
  
\begin{rem}\label{chain}
Let $R(d)$ be the set of roots $\al_{i,j}$ with $1\le i\le d\le j<n$.
Consider the following partial order on $R(d)$ : 
$\al_{i_1,j_1}\ge \al_{i_2,j_2}$ if $i_1\ge i_2$ and $j_1\ge g_2$. Then the Dyck paths are exactly the maximal chains in the poset $R(d)$.
\end{rem}

The monomials $f^\bs\ell_{m\om_d} = \prod f_{i,j}^{s_{i,j}} \ell_{m\om_d}$
satisfying the conditions 
\begin{equation}\label{FFLV}
\sum_{\al\in \bp} s_\al\le m \quad \text{ for all Dyck paths } \bp
\end{equation}
form a basis of $L_{m\om_d}$. 
In other words, the basis elements are labeled by the integer points in the scaled 
Stanley chain polytope of $R(d)$ (see \cite{Stan}). 
Let $P_{m\om_d}\subset \bR_{\ge 0}^{d(n-d)}$
denote the FFLV polytopes defined by inequalities \eqref{FFLV} and let 
$S_{m\om_d}=P_{m\om_d}\cap \bZ_{\ge 0}^{d(n-d)}$. 

\begin{example}\label{fundamental}
Let $m=1$. Then the points 	in $S_{\om_d}$ are in one-to-one correspondence
with antichains in $R(d)$. In fact, the antichains in $R(d)$ are of the form
\[
\al_{i_1,j_1}, \dots, \al_{i_s,j_s} \text{ with }  i_1<\dots <i_s, j_1>\dots >j_s
\]
and the FFLV basis in $L_{\omega_d}$ is 
\[
f_{i_1,j_1} \dots f_{i_s,j_s}\ell_{\om_d},\quad  i_1<\dots <i_s, j_1>\dots >j_s.
\]
\end{example} 

Let $\fp_d\subset\msl_n$ be the standard maximal parabolic subalgebra corresponding to 
the radical $\fa_d$. More precisely,  $\fp_d$ is spanned by the matrix units
$E_{i,j}$ such that either $i\le d$ or $j>d$ (with an extra traceless condition). 
One has $\msl_n = \fa_d\oplus \fp_d$. Since $\fp_d$ is a subalgebra
in $\msl_n$, any representation of $\msl_n$ is naturally a module over the parabolic subalgebra. 
One easily sees that 
$$\fp_d L_{m\om_d}(r) \subset \bigoplus_{r'\le r} L_{m\om_d}(r')$$ 
(because $L_{m\omega_d}(r)=\mathrm{span} \{f_{i_1,j_1}\dots f_{i_r,j_r}\ell_{m\om_d}, i_\bullet\le d\le j_\bullet\}$ and $\fp_d \ell_{m\om_d}=0$)  
and hence we obtain an induced $\fp_d$ action on the module
\[
{\rm gr} L_{m\omega_d}=\bigoplus_{r'\le r} L_{m\om_d}(r')\Bigl/\Bigr.\bigoplus_{r'< r} L_{m\om_d}(r').
\]

\begin{rem}
We note that ${\rm gr} L_{m\omega_d}$ is naturally isomorphic to 
$L_{m\omega_d}$ as a vector space and as $\fa_d$ modules.
We use ${\rm gr}$ to 
stress that the  $\fp_d$ action is different. In particular, 
${\rm gr} L_{m\omega_d}$ is a direct sum of $\fp_d$ submodules
$L_{m\omega_d}(r)$, while the summands
$L_{m\omega_d}(r)$ of $L_{m\omega_d}$ are not $\fp_d$ invariant.
\end{rem}

\begin{lem}
The induced actions of $\fa_d$ and $\fp_d$ on ${\rm gr} L_{m\omega_d}$
glue into the action of the degenerate algebra $\msl_n^{(d)}$, which is 
isomorphic to $\fa_d\oplus \fp_d$ as a vector space. The Lie algebra structure  
on $\msl_n^{(d)}$ is defined as follows: $\fp_d$ is a subalgebra,
$\fa_d$ is an abelian ideal and the adjoint action of $\fp_d$ on $\fa_d$ is induced
by the identification of $\fa_d$ with the natural $\fp_d$ module $\msl_n/\fp_d$.
\end{lem}
\begin{proof}
Immediate from the definitions (see also \cite{FFL1,F3,PY}). 
\end{proof}

\subsection{The extension}	
Let us consider the following graded module $V$ of the abelian  Lie
algebra $\fa_d$. As a vector space $V=V(0)\oplus V(1)$, where $V(0)$ is one-dimensional	space spanned by vector $v$ and $V(1)\simeq \fa_d$ as a vector space. The abelian
Lie algebra $\fa_d$ acts trivially on $V(1)$ and sends $V(0)$ to $V(1)$ via
the identification $V(0)\otimes \fa_d\simeq V(1)$. Then $V$ is generated
from $v$ by the action of $\fa_d$ and $V$ is isomorphic
to the quotient of the polynomial algebra in variables $f_{i,j}$, $1\le i\le d\le j\le n-1$ by the ideal generated by all monomials of degree two.

Similarly, we define the $\fa_d$ modules $V_M$, $M\ge 1$ as the quotient 
of $\bC[f_{i,j}]$ by the ideal generated by all degree $M+1$ monomials
(in particular, $V\simeq V_1$).  
Then $V_M \simeq V^{\odot M}_1$ as $\fa_d$ modules and   
\[
V_M = \bigoplus_{r=0}^M V_M(r),\  V_M(r)\simeq S^r(\fa_d).
\]	
Let $v_M\in V_M$ be the cyclic vector, $v_M=v^{\T M}$ with respect to the embedding $V_M\subset V_1^{\T M}$. 

\begin{lem}
The $\fa_d$ action on $V_M$ extends to the $\msl_n^{(d)}$ action by letting
$\fp_d$ act on  $V_M(r)$ as on $S^r(\msl_n/\fp_d)$.
\end{lem}
\begin{proof}
We note that $\fa_d$ acts on $V_M$ simply by multiplication. Now the claim follows from the definition of the action of $\fp_d$.
\end{proof}

We are now ready to define the main algebraic object of this paper.

\begin{dfn}
For $m,M\ge 0$ we define the $\fa_d$ module 
\[
L_{m,M} = \mathrm{gr}L_{m\om_d}\odot V_M.
\]	
\end{dfn}

\begin{lem}
$L_{m,M}$ are cyclic $\fa_d$ modules, $L_{m,M}=\mathrm{gr} L_{\om_d}^{\odot m}\odot V^{\odot M}$. The  $\fa_d$ action extends to the action of $\msl_n^{(d)}$.
\end{lem}
\begin{proof}
The first claim is a reformulation of the definition. The second claim holds
true since all the factors are $\msl_n^{(d)}$ modules. 	
\end{proof}

In what follows we denote by $\ell_{m,M}\in L_{m,M}$ the cyclic vector;
$\ell_{m,M} = \ell_{m\om_d}\T v_M$. 

\begin{rem}
We  show in the next section that the spaces $L_{m,M}$ are closely
related to the closure in $\bP^{d(n-d)}\times \Gr(d,n)$ of the 
graph $\G(d,n)$ of the birational exponential  map  $\bP^{d(n-d)}\to {\rm Gr}(d,n)$.   
\end{rem}

\subsection{Bases}
We consider the real vector space $\bR^{d(n-d)}$ whose coordinates are labeled
by the roots $\al_{i,j}$, $1\le i\le d\le j<n$ of $\fa_d$, or simply by the 
corresponding pairs $i,j$.
\begin{dfn}\label{polytopes}
Let $m,M\ge 0$.
The polytope $X(m,M)\subset\bR_{\ge 0}^{d(n-d)}$ is the set of
 points
$(s_{i,j})$ with nonnegative coordinates satisfying the following inequalities
for all $r\ge 1$:
\begin{equation}\label{inequalities}
\sum_{\al\in \bp_1\cup\ldots\cup \bp_r} s_\al\le rm+M 
\quad \text{ for all tuples of Dyck paths } \bp_1,\dots,\bp_r.
\end{equation}
We set $S({m,M})=X(m,M)\cap \bZ^{d(n-d)}$.
 \end{dfn}

\begin{rem}
One has an obvious inclusion $X({m,M})\subset X(m+M,0)$, where $X(m+M,0)$ coincides 
with the FFLV polytope of weight $(m+M)\om_d$. 	
\end{rem}

\begin{lem}
In Definition \ref{polytopes} it suffices to consider inequalities 
\eqref{inequalities} with $r\le\min(d,n-d)$.  
\end{lem}	
\begin{proof}
We note that if $r=	\min(d,n-d)$, then there exist $r$ Dyck paths 
$\bp_1,\dots,\bp_r$ such that $\bigcup_{i=1}^r \bp_i$ cover the whole set of roots
$\al_{i,j}$, $i\le j\le d$. In fact, without loss of generality we assume that $d\le n-d$
(i.e.\ $r=d$). Then we take
\[
\bp_i=\{\al_{1,d},\dots,\al_{i,d},\al_{i,d+1},\dots,\al_{i,n-1},\al_{i+1,n-1},
\dots, \al_{d,n-1}\}.
\] 
The inequality corresponding to the paths $\bp_1,\dots,\bp_r$ reads as
$\sum_{i\le d\le j} s_{i,j}\le md+M$, which implies all inequalities 
\eqref{inequalities} for $r>d$. 
\end{proof}

\begin{rem}
If $M=0$ then all inequalities \eqref{inequalities} are implied by
the $r=1$ inequalities.
\end{rem}

Our goal is to show that the following properties hold:
\begin{itemize}
	\item $f^\bs \ell_{m,M}$, $\bs\in S(m,M)$ form a basis of $L_{m,M}$;
	\item $S(m,M)+S(m',M')=S(m+m',M+M')$, where $+$ is the Minkowski addition;
	\item $L_{m,M}$ is cyclic $\msl_n^{(d)}$ module with defining relations \begin{equation}\label{relations}
		f_{i_1,j_1}^{b_1}\dots f_{i_r,j_r}^{b_r}\ell_{m,M}=0\ \text{ if }\ b_1+\dots+b_r>rm+M.
	\end{equation}	
\end{itemize}

\begin{rem}\label{notonly}
Our goal is to show that relations \eqref{relations} are defining relations 
of $L_{m,M}$ as 
$\msl_n^{(d)}$ module. However, they are not defining if $L_{m,M}$
is considered only as an $\fa_d$ module. In other words,
the defining relations of the $\fa_d$ module  $L_{m,M}$ are obtained from
\eqref{relations} applying the universal enveloping algebra $\U(\fp_d)$. 	
\end{rem}

\begin{lem}
The three properties above  hold true for $M,M'=0$.
\end{lem}
\begin{proof}
The polytope $X(m,0)$ is the FFLV polytope for the weight $m\om_d$.
Hence the first two properties are direct corollaries from \cite{FFL1}.
To prove the third property we note that if $M=0$ then all the relations are implied by the $r=1$ relations. In fact, all $f_{i,j}$ commute and 
$b_1+\dots+b_r>rm$ implies that $b_k>m$ for some $k=1,\dots,r$.
However it was shown in \cite{FFL1} that the defining relations of 
${\rm gr} L_{m\om_d}$ are $f_{i,j}^{m+1}\ell_{m\om_d}=0$ (in fact, it suffices 
to take $f_{1,n-1}^{m+1}\ell_{m\om_d}=0$).   	
\end{proof}

\begin{lem}\label{defrel}
The relations 
$$f_{i_1,j_1}^{b_1}\dots f_{i_r,j_r}^{b_r}\ell_{m,M}=0 \text{ if } b_1+\dots+b_r>rm+M$$
hold true in $L_{m,M}$.
\end{lem}
\begin{proof}
Recall that  $\ell_{m,M}=\ell_{\om_d}^{\T m} \T v_M \subset L_{m\omega_d}\odot V_M$. 
The claim of the Lemma is implied by two observations:
\begin{itemize}
\item $f^\bs v_M=0 \text{ if } \sum s_\al>M$,
\item $f_{i,j}^2 \ell_{\om_d} = 0$.	
\end{itemize}	
In fact, when applying $f_{i_1,j_1}^{b_1}\dots f_{i_r,j_r}^{b_r}$ to $\ell_{m,M}$ 
one distributes the factors of the monomial among the factors of the tensor
product $\ell_{\om_d}^{\T m}\T v_M$. Since $f^\bs v_M=0$ if $\sum s_\al>M$,
it suffices to show that 
$f_{i_1,j_1}^{b_1}\dots f_{i_r,j_r}^{b_r}  \ell_{\om_d}^{\T m}$ vanishes, provided $b_1+\dots+b_r>rm$.
In this case $b_i>m$ for some $i$. However, 
$f_{i,j}^2 \ell_{\om_d} = 0$ and hence $f_{i,j}^{m+1} \ell_{\om_d}^{\T m} = 0$.
\end{proof}

\begin{prop}\label{span}
The vectors $f^\bs\ell_{m,M}$, $\bs\in S(m,M)$ span $L_{m,M}$.
\end{prop}
\begin{proof}
Thanks to Lemma \ref{defrel} it is enough to show that any vector 
$f^\bt\ell_{m,M}$, $\bt\notin S(m,M)$ can be rewritten as a linear combination
of vectors $f^\bs\ell_{m,M}$, $\bs\in S(m,M)$ using relations from Lemma 
\ref{defrel} (as we mentioned in Remark \ref{notonly} we use all the relations
\eqref{relations} and their $\fp_d$ consequences). 

Recall the idea of the proof of a similar statement from \cite{FFL1}, which
is the $M=0$, $r=1$ case.
Let us (totally) order the root vectors as follows:
\begin{equation}\label{order}
f_{d,n-1}\succ f_{d-1,n-1}\succ\dots  \succ f_{1,n-1} \succ f_{d,n-2}\succ\dots \succ f_{1,n-2} \succ\dots\succ f_{1,d}.
\end{equation} 
In other words, $f_{i,j}\succ f_{i',j'}$ if and only if $j>j'$ or 
($j=j'$ and $i>i'$); note that this total order refines the partial order from Remark \ref{order}.
In what follows we consider the associated homogeneous 
(with respect to total degree -- the sum of all exponents) lexicographic order
on the set of monomials $f^\bt$.
It was shown in \cite{FFL1} that if $\bt\notin S(m,0)$, then $f^\bt\ell_{m,0}$
can be rewritten as a linear combination of smaller monomials (of the same total degree). 
Since there are obviously only finitely many monomials smaller than a given one, we obtain that $f^\bs\ell_{m,0}$, $\bs\in S(m,0)$ span the space $L_{m,0}$. We generalize this argument to the case of general $M$.

We prove that 
if $\bt$ does not satisfy some inequality from  \eqref{inequalities}, then 
$f^\bt$ can be rewritten as a linear combination of smaller monomials (with
respect to the order induced by \eqref{order}).
Assume we are given an $r$-tuple of Dyck paths $\bp_1,\dots,\bp_r$. 
We assume that the union of these paths is of width $r$ (i.e.\ the union can not be covered by a smaller number of paths).
We represent 
$\bp_1 \cup \dots\cup \bp_r$  as  a union of $r$ pairwise nonintersecting sets $\overline{\bp}_1,\dots,\overline{\bp}_r$ with the following properties:
\begin{itemize}
\item each $\overline{\bp}_k$ is a dense chain (i.e.\ the are no elements between two neighbors of $\overline{\bp}_k$,
\item $\al_{1,d}\in \bp_1$, $\al_{d,n-1}\in \bp_1$,
\item there exist numbers $1=i_1<\dots <i_r\le d$ 
such that $\al_{i_k,d}\in \overline{\bp}_k$ and $i_k$ is the smallest number with this property 
(for all $k\in [r]$),
\item
there exist numbers  $n-1=j_1>\dots > j_r\ge d$ such that $\al_{d,j_k}\in \overline{\bp}_k$
$j_k$ is the largest number with this property 
(for all $k\in [r]$).
\end{itemize}
Now let $b_k=\sum_{\al\in\overline{\bp}_k} s_\al$.
Then by assumption $b_1+\dots +b_r > mr+M$ and hence
by Lemma \ref{defrel} one has: $f^{b_1}_{i_1,j_1}\dots f^{b_r}_{i_r,j_r}\ell_{m,M}=0$. 
Now our goal is to show that $\U(\msl_n^{(d)})f^{b_1}_{i_1,j_1}\dots f^{b_r}_{i_r,j_r}\ell_{m,M}$
contains a vector which is equal to $f^\bt\ell_{m,M}$ plus some linear combination of smaller monomials applied to $\ell_{m,M}$. 

We follow the strategy of \cite{FFL1}. More precisely, we 
consider the derivations $\pa_\al$ of the polynomial ring in variables $f_{\be}$ defined by  $\pa_\al f_\be = f_{\be-\al}$ if $\be-\al$ is a root 
and zero otherwise. We denote by $\pa_{i,j}$ the operators $\pa_{\al_{i,j}}$. 
We start with $f^{b_r}_{i_r,j_r}$ and the Dyck path $\overline{\bp}_r$. Applying certain linear combination of the products of operators $\pa_{\al}$ to
$f^{b_r}_{i_r,j_r}$ we obtain $\prod_{\al\in \overline{\bp}_r} f_\al^{t_\al}$ plus linear combination of smaller monomials.
An important thing is that we only use the operators of the form 
$\pa_{\bullet,j_r }$ and $\pa_{i_r, \bullet}$. Therefore,   
since $i_1<\dots<i_r$ and $j_1>\dots >j_r$, the operators $\pa_\al$ we use at this stage do not affect $f^{b_k}_{i_k,j_k}$ for $k<r$.
As a result we obtain that 
\[
\prod_{k=1}^{r-1} f^{b_k}_{i_k,j_k}
\prod_{\al\in  \overline{\bp}_r} f_\al^{t_\al}\ell_{m,M}
\]
can be expressed as a linear combination of smaller monomials applied to $\ell_{m,M}$.

We then proceed with $f^{b_{r-1}}_{i_{r-1},j_{r-1}}$ and the Dyck path $\overline{\bp}_{r-1}$. As in the previous case, we apply the derivation 
operators to $f^{b_{r-1}}_{i_{r-1},j_{r-1}}$ and obtain $\prod_{\al\in \overline{\bp}_{r-1}} f_\al^{t_\al}$ modulo a linear combination 
of smaller terms. The derivation we use at this stage do not affect 
$f^{b_a}_{i_a,j_a}$ for $a<r-1$ and do not affect the result of application of derivation operators to $f^{b_r}_{i_r,j_r}$ from the previous stage. As a result at this point we obtain that 
\[
\prod_{k=1}^{r-2} f^{b_k}_{i_k,j_k}
\prod_{\al\in \overline{\bp}_{r-1} \cup  \overline{\bp}_r} f_\al^{t_\al}\ell_{m,M}
\]
can be expressed as a linear combination of smaller monomials applied to $\ell_{m,M}$.

We proceed in the same way with $f^{b_{r-2}}_{i_{r-2},j_{r-2}}$, etc. and 
as a result obtain that
\[
\prod_{\al\in \overline{\bp}_1 \cup \dots \cup  \overline{\bp}_r} f_\al^{t_\al}\ell_{m,M}
\]
can be expressed as a linear combination of smaller monomials applied to $\ell_{m,M}$. Hence the same property holds true for the whole monomial
$f^\bt \ell_{m,M}$, as desired.
\end{proof}	

\begin{thm}
Vectors $f^\bs \ell_{m,M}$, $\bs\in S(m,M)$ form a basis of $L_{m,M}$ and
relations \eqref{relations}
are defining for the cyclic $\msl_n^{(d)}$ module $L_{m,M}$.
\end{thm}	
\begin{proof}
We use the results of \cite{FFL3}. 
The bases $f^\bs \ell_{1,0}$, $\bs\in S(1,0)$ and $f^\bs \ell_{0,1}$, $\bs\in S(0,1)$ are essential bases with respect to the total order 
\eqref{order}. 
Since $L_{m,M}$ is defined as cyclic product of 
$L_{1,0}$ ($m$ times) and $L_{0,1}$ ($M$ times), 
we know (see Theorem \ref{Msum})  that the Minkowski sum $S(m_1,M_1)+S(m_2,M_2)$ is equal to $S(m_1+m_2,M_1+M_2)$ for any $m_1,m_2,M_1,M_2$.
Hence \cite{FFL3} imply that vectors $f^\bs \ell_{m,M}$, $\bs\in S(m,M)$ are linearly independent in $L_{m,M}$. Thanks to Proposition \ref{span} we 
obtain the first claim of our theorem. To prove the second claim, recall that 
we've shown in the proof of Proposition \ref{span} that relations \eqref{relations}
allow to rewrite any monomial in terms of $f^\bs \ell_{m,M}$, $\bs\in S(m,M)$.
Hence these are the defining relations. 
\end{proof}

\section{Geometry}\label{geometry}
Let us define the main geometric object we are interested in.
We consider the Grassmannian $\Gr(d,n)$ of $d$-dimensional subspaces
of an $n$-dimensional vector space. This is a smooth projective algebraic
variety of dimension $d(n-d)$. The Grassmannian admits an affine paving by
Bruhat cells; in particular, there is natural exponential map $\imath: \bA^{d(n-d)}\to \Gr(d,n)$
from the affine space to the Grassmannian, whose image is the 
open cell. We extend the map $\imath$ to the birational map
$\bP^{d(n-d)}\to \Gr(d,n)$ and denote by $\G(d,n)$ the closure of its graph (here we consider $\bA^{d(n-d)}$ as an open cell of $\bP^{d(n-d)}$). More precisely,
we have the following definition:

\begin{dfn}
The variety $\G(d,n)$ is a subvariety of $\bP^{d(n-d)}\times \Gr(d,n)$
defined as the closure of the set $\{(x,\imath(x)), \ x\in \bA^{d(n-d)}\}$. 
\end{dfn}

\begin{example}
If $d=1$, then  $\Gr(1,n)\simeq \bP^{n-1}$ and  $\G(1,n)\simeq \bP^{n-1}$ is 
embedded diagonally  into $\bP^{n-1}\times \Gr(1,n)$.
\end{example}

Using  Pl\"ucker embedding, one sees that the graph $\G(d,n)$ is naturally
embedded into the product of projective spaces $\bP^{d(n-d)}\times
\bP(\Lambda^d(L))$, where $L$ is the $n$-dimensional vector space with basis $\ell_i$, $i\in [n]$. Let $X_{I}=X_{i_1,\dots,i_k}$ be the homogenenous 
coordinates on $\bP(\Lambda^d(L))$ and $Y_0$, $Y_{i,j}$, 
$1\le i\le d\le j< n$ be   the homogenenous 
coordinates on $\bP^{d(n-d)}$. Then
the corresponding homogeneous coordinate ring $H$ of 
$\G(d,n)\subset \bP^{d(n-d)}\times \bP(\Lambda^d(L))$ is the quotient 
of the polynomial ring in variables $X_I$, $Y_0$, $Y_{i,j}$ by the
bi-homogeneous ideal $\I(d,n)$. 
In particular,  
$H=\bigoplus_{m,M\ge 0} H_{m,M}$, where $m$ is the total degree with respect
to the $X$-variables and $M$ is the total degree with respect to the $Y$-variables. One of our goals is to describe the spaces $H_{m,M}$.

\subsection{The group action}

Recall the Lie algebra $\msl_n^{(d)}=\fp_d\oplus\fa_d$ and its representations
$L_{m,M}$. Let $P_d\subset SL_n$ be the standard parabolic subgroup with the
Lie algebra $\fp_d$ and let $N_d\subset SL_n$ be the unipotent subgroup with 
the Lie algebra $\fa_d$. 
Recall that $\fa_d$ is endowed with the structure of $\fp_d$ module and hence
 of the $P_d$ module. Hence one obtains the group
\[
SL_n^{(d)} = P_d\ltimes \exp(\fa_d) = P_d\ltimes  N_d.
\] 
In particular, $SL_n^{(d)}$ contains a normal abelian subgroup $N_d$ and a parabolic
subgroup $P_d$. 

The following observation is simple, but crucial in our approach. Recall the 
cyclic vector $\ell_{m,M}\in L_{m,M}$. In what follows we denote by
$[\ell_{m,M}]\in\bP(L_{m,M})$ the line containing the cyclic vector.

\begin{constr}
For $m,M>0$ the orbit closure $\overline{N_d\cdot[\ell_{m,M}]} 
\subset \bP(L_{m,M})$ is isomorphic to
$\G(d,n)$. 
\end{constr}
\begin{proof}
Let us denote $\overline{N_d\cdot[\ell_{m,M}]} 
\subset \bP(L_{m,M})$ by $\G_{m,M}$. 
Recall that  $L_{m,M}=L_{1,0}^{\odot m} \odot L_{0,1}^{\odot M}$
and hence $\G_{m,M}$ sits inside $\G_{m,0}\times \G_{0,M}$
as the closure of the orbit passing through $[\ell_{m,0}]\times [\ell_{0,M}]$.
Since $\G_{m,0}$ sits diagonally inside $\G_{1,0}^{\times m}$ and
$\G_{m,0}$ sits diagonally inside $\G_{0,1}^{\times M}$, we know that
$\G_{m,M}\simeq \G_{1,1}$. We are left to show that $\G_{1,1}\simeq \G(d,n)$.

Let $a\in \fa_d$. Recall that $L_{1,0}= \gr \Lambda^d(L)$, $L_{0,1}= \mathrm{span} \{v\} \oplus \fa_d$. Then 
\[
\exp(a)([\ell_{1,0}]\times [\ell_{0,1}]) = [\exp(a)\ell_{\om_d} ]
\times [v + a] = [(\ell_1+a\ell_1)\wedge\dots\wedge  (\ell_d+a\ell_d)]\times [v+a].
\]
Points $[v + a]\in \bP^{d(n-d)}$, $a\in\fa_d$  form the open dense cell
$\bA^{d(n-d)}\subset \bP^{d(n-d)}$ and  by definition
$[(\ell_1+a\ell_1)\wedge\dots\wedge  (\ell_d+a\ell_d)] =\imath(a)$
(the left hand side is the standard parametrization of the open cell in the 
Grassmannian). Hence $\G_{1,1}\simeq \G(d,n)$.
\end{proof}

\begin{cor}
The abelian unipotent group $N_d\simeq \bG_a^{d(n-d)}$ acts on $\G(d,n)$ with
an open dense orbit. The $N_d$ action on $\G(d,n)$ extends to the action of the group
$SL_n^{(d)}$.	
\end{cor}

\subsection{The fibers}
By definition we have a projection map $\varphi:\G(d,n)\to \Gr(d,n)$. Our goal
is to describe the fibers of the map $\varphi$. To this end, we consider the following
stratification of the Grassmannians. Recall that we have fixed $n$-dimensional 
space $L$, the points of $\Gr(d,n)$ are $d$ dimensional subspaces of $L$.
\begin{definition}
Let $X_k\subset \Gr(d,n)$ be the set of subspaces $U\subset L$ such that 
\[
\dim \left(U\cap \mathrm{span} \{\ell_{d+1},\dots,\ell_n\}\right) = k.
\] 
\end{definition}

By definition,  $X_k$ is nonempty if and only if $k\le \min(d,n-d)$.

\begin{rem}
Each $X_k$ is a union of several Schubert cells in the Grassmannian.
More precisely, the Schubert cells $C_I\subset \Gr(d,n)$ are labeled by subsets 
$I\subset [n]$ such that $|I|=d$. Then $X_k$ is the union of the cells 
$C_I$ such that $|I\cap \{d+1,\dots,n\}|=k$.
In particular, $X_0$ is the open cell.     
\end{rem}

We adjust the map $\imath$ and the chosen basis $\ell_1,\dots,\ell_n$ of $L$ in such a way that  
$X_0$ is exactly the image of the map $\imath: \bA^{d(n-d)}\to \Gr(d,n)$.  One has
\[
\Gr(d,n)=\bigsqcup_{k=0}^{\min(d,n-d)} X_k,\ \overline{X_k}=\bigsqcup_{k'\ge k} X_{k'}.
\] 

\begin{example}
If $d<n-d$, then $X_d\simeq \Gr(d,n-d)$. If $d>n-d$, then $X_{n-d}\simeq \Gr(2d-n,d)$.
If $d=n-d$, then $X_d$ is a point.
\end{example}

Recall the abelian Lie algebra $\fa_d$ (abelian unipotent radicalin $\msl_n$) and
its cyclic module $V$, which is the direct sum of a one-dimensional space $V(0)$
(spanned by the cyclic vector) and the space $V(1)$ isomorphic to $\fa_d$ as a vector space. In particular, by definition $\G(d,n)$ sits inside $\bP(V)\times \Gr(d,n)$. 

Recall that $\G(d,n)$ is the closure of the open cell 
$\G^\circ(d,n)=\{(x,\imath(x)), \ x\in \bA^{d(n-d)}\}$.
\begin{rem}\label{X=Y}
Let $a=\sum a_{i,j} f_{i,j}\in\fa_d$. Then the value of $Y_{i,j}$ on 
$\exp(a)\ell_{0,1}$ and the value of $X_{[d]\setminus\{i\}\cup\{j+1\}}$ are equal to $a_{i,j}$
(note that $Y_0$ and $X_{[d]}$ are equal to one on $\G^\circ(d,n)$). 
\end{rem}

\begin{lem}	\label{boundary}
The boundary $\G(d,n)\setminus \G^\circ(d,n)$ inside $\bP(V)\times \Gr(d,n)$ 
belongs to $\bP(V(1))\times \bigsqcup_{k\ge 1} X_k$.  
\end{lem}
\begin{proof}
All the points in $\G^\circ(d,n)$ satisfy the following equations for all
$i\in [d]$, $j=d+[n-d]$:
\[
Y_{0}X_{[d]\setminus \{i\}\cup j} - (-1)^{d-i} Y_{i,j-1} X_{[d]} = 0 
\]
(see Remark \ref{X=Y}).
Hence $Y_0=0$ implies $X_{[d]}=0$. 

Now assume $X_{[d]}=0$. Then we claim that
$Y_{0}$ also vanishes on a point of $\G(d,n)$. If $Y_{0}\ne 0$, then
$X_{[d]\setminus \{i\}\cup j}=0$ for all  $i\in [d]$, $j\in d+[n-d]$.
Let us show that actually the condition $Y_{0}\ne 0$ implies vanishing of all
the Pl\"ucker coordinates $X_I$. In fact, assume $I=[d]\setminus A\cup B$,
where $A\subset [d]$, $B\subset d+[n-d]$. Let $a\in A$. Then among Pl\"ucker relations
one has (see e.g. \cite{Fu})
\[
X_{[d]}X_I = \sum_{b\in B}\pm X_{[d]\setminus\{a\}\cup b} X_{I\cup \{a\}\setminus \{b\}}
\]
and hence the following polynomial vanishes on $\G^\circ(d,n)$
\[
Y_{0}X_I = \sum_{b\in B}\pm Y_{a,b-1} X_{I\cup \{a\}\setminus \{b\}}.
\] 
By induction on the cardinality of $B$ we may assume that the right hand side 
vanishes. Hence if $Y_0\ne 0$, then all $X_I=0$, which is not possible.   

To conclude the proof we note that the zero set of $Y_0$ in $\bP(V)$
is $\bP(V(1))$ and the zero set of $X_{[d]}$ in $\Gr(d,n)$ is  
$\bigsqcup_{k\ge 1} X_k$.
\end{proof}	

Let us fix the decomposition 
$L=L^-\oplus L^+$, where 
\[
L^-=\mathrm{span}\{\ell_1,\dots,\ell_d\}, \ L^+=\mathrm{span}\{\ell_{d+1},\dots,\ell_n\}.
\]
Then $V(1)$ and $\fa_d$ are naturally identified with the space 
$\mathrm{Hom}(L^-,L^+)$ of linear maps from $L^-$ to $L^+$. 
Hence Lemma \ref{boundary} gives a natural embedding $\G(d,n)\setminus \G^\circ(d,n)\subset 
\bP(\mathrm{Hom}(L^-,L^+))\times \Gr(d,n)$. 

Let $\mathrm{pr}^-$ be the projection from $L$ to $L^-$ along $L^+$.
Recall the projection $\varphi:\G(d,n)\to\Gr(d,n)$. For a vector $f$ we denote by 
$[f]=\mathrm{span}\{f\}$ the line in the corresponding projective space. 

\begin{thm}\label{main}
For $U\in\Gr(d,n)$ the preimage $\varphi^{-1}(U)$ is a projective subspace in 
$\bP(\mathrm{Hom}(L^-,L^+))$. A pair $([f],U)$ belongs to 
$\varphi^{-1}(U)$ for a linear map $f$  
if and only if the following conditions hold:
\[
\mathrm{Im} f \subset U\cap L^+,\qquad \mathrm{pr}^-(U) \subset \ker f .
\] 
\end{thm}

We first prove this theorem for $U$ being a span of basis vectors $\ell_i$ and then deduce the general case.
Let us fix $J\in\binom{[n]}{d}$ and let $U_J=\mathrm{span}\{\ell_j,j\in J\}$.

\begin{prop}\label{fiberoverfp}
$([f],U_J)\in \varphi^{-1}(U_J)$ if and only if 
\begin{equation}\label{fcond}
\mathrm{Im} f \subset \mathrm{span}\{\ell_j, j\in J_{>d}\},\qquad 
f(\ell_j)=0 \text{ for any } j\in J_{\le d}.
\end{equation}
\end{prop}
\begin{proof}
We first show that if $f$ satisfies conditions \eqref{fcond}, then $([f],U_J)\in \varphi^{-1}(U_J)$. Our assumption implies the existence of the decomposition
\[
f=\sum_{\substack{a\in [d]\setminus J_{\le d}\\ b\in J_{>d}}} z_{b,a} E_{b,a},
\ z_{b,a}\in\bC
\]
for some coefficients $z_{b,a}$. Let us take the following family of points 
$p(t)\in \bP(V)\times \Gr(d,n)$ (recall that $v\in V$ is the spanning vector of $V(0)$):
\[
p(t)= \Bigl([v+tf], \mathrm{span}\{\ell_j, j\in J_{\le d}\} \oplus \mathrm{span}\{\ell_{a}+tf(\ell_a), a\in [d]\setminus J\}\Bigr), \quad t\in\bC. 
\]
Then $\lim_{t\to\infty} [v+tf] = [f]$, and if $\det z_{b,a}\ne 0$, then
\begin{multline*}
\lim_{t\to\infty}\bigl(\mathrm{span}\{\ell_j, j\in J_{\le d}\} \oplus \mathrm{span}\{\ell_{a}+tf(\ell_a), a\in [d]\setminus J\}\bigr)=\\
\mathrm{span}\{\ell_j, j\in J_{\le d}\} \oplus \mathrm{span}\{\ell_{b}, 
b\in J_{> d}\} = U_J.
\end{multline*}
Therefore, if $\det z_{b,a}\ne 0$, then $([f],U_J)\in \varphi^{-1}(U_J)$.
Since the preimage of a point is closed, the condition $\det z_{b,a}\ne 0$
can be omitted.

Now let us show that if $([f],U_J)\in \varphi^{-1}(U_J)$, then $f$ satisfies
\eqref{fcond}. We write
\[
f=\sum_{{a\in [d], b\in d+[n-d]}} z_{b,a} E_{b,a}
\]
and show that $z_{b,a}=0$ unless $a\in [d]\setminus J_{\le d}$ and 
$b\in J_{>d}$.

Let us take $b>d$, $b\notin J$. Then for any $a\in [d]$ one has the following Pl\"ucker relation (see e.g.\cite{Fu}):
\[
X_{[d]\setminus\{a\}\cup\{b\}} X_J = 
\pm \delta_{a\in J} X_{[d]} X_{J\setminus\{a\}} + 
\sum_{j\in J_{>d}}\pm X_{[d]\setminus\{a\}\cup\{j\}} X_{J\setminus\{j\}\cup\{b\}},   
\] 
where   $\delta_{a\in J}=1$ if $a\in J$ and vanishes otherwise.
Hence the following relations hold  on $\G(d,n)$:
\[
Y_{a,b-1} X_J = 
\pm \delta_{a\in J} Y_0 X_{J\setminus\{a\}} + 
\sum_{j\in J_{>d}}\pm Y_{a,j-1}X_{J\setminus\{j\}\cup\{b\}}.   
\]  
Evaluating this relation at the point $(f,U_J)$ one gets $z_{b,a}=Y_{a,b-1}(f)=0$
(since $X_{I}$ vanishes on $U_J$ unless $I=J$).

Now let us take $a\in J_{\le d}$. Then for any $b=d+1,\dots,n$ 
one has the following Pl\"ucker relation:
\[
X_J X_{[d]\setminus\{a\}\cup\{b\}}  = 
\delta_{b\notin J}\pm X_{J\setminus\{a\}\cup\{b\}} X_{[d]} +
\sum_{j\in [d]\setminus J_{\le d}} X_{J\setminus\{a\}\cup\{j\}}X_{[d]\setminus\{j\}\cup\{a,b\}},
\]
where   $\delta_{b\notin J}=1$ if $b\notin J$ and vanishes otherwise.
Hence the following relations hold  on $\G(d,n)$ (see Remark \ref{X=Y}):
\[
X_J Y_{a,b-1}  = 
\delta_{b\notin J}\pm X_{J\setminus\{a\}\cup\{b\}} Y_0 +
\sum_{j\in [d]\setminus J_{\le d}} X_{J\setminus\{a\}\cup\{j\}}Y_{j,b-1}.
\]  
Evaluating this relation at $([f],U_J)$ one gets $z_{b,a}=Y_{a,b-1}(f)=0$.
\end{proof}

We are now ready to prove Theorem \ref{main}.
\begin{proof}[Proof of Theorem \ref{main}]
Let us fix a point $U\in \Gr(d,n)$, $U\notin X_0$. Let $J$ be a $d$-tuple such that $U$ belongs to the Schubert cell with the center at $U_J$. 
Let $N\subset SL_n$ be the unipotent lower triangular subgroup (in particular,
$N\supset N_d$). Then there exists $g\in N$ such that $U=gU_J$. 

Recall the maximal parabolic subgroup $P_d\subset SL_n$. Then 
$P_d\cap N = N^{(1)}\times N^{(2)}$, where $N^{(1)}\subset SL_{d}$ and
$N^{(2)}\subset SL_{n-d}$ are unipotent subgroups. Clearly $N=N^{(1)}N^{(2)}N_d$ and hence  there exist elements $g^{(1)}\in 	N^{(1)}$, $g^{(2)}\in 	N^{(2)}$
and $g_d\in N_d$ such that $U=g^{(1)}g^{(2)}g_d U_J$. 

The group $N$ is naturally embedded into $SL_n^{(d)}$. In fact,
$SL_n^{(d)}=P_d\ltimes N_d$, $N^{(1)}\times N^{(2)}\subset P_d$ and 
$N=(N^{(1)}\times N^{(2)})\ltimes N_d$. Hence $N$ acts on $\G(d,n)$ and
the natural projection map  $\varphi: \G(d,n)\to \Gr(d,n)$ is $N$-equivariant.

We claim that for any $J\ne [d]$ one has 
\begin{equation}\label{shift}
g\bigl(\varphi^{-1}U_J\bigr) =
\{f:\ \mathrm{Im} f \subset (gU_J)\cap L^+,\ \mathrm{pr}^-(gU_J) \subset \ker f \}\times U.
\end{equation} 
Given an $f_J\in \fa_d\simeq \mathrm{Hom}(L^-,L^+)$ one has     
$g^{(1)}g^{(2)}g_d(f) = g^{(2)} \circ f \circ (g^{(1)})^{-1}$
(here we consider $g^{(1)}$ as an element of $\mathrm{Aut}(L^-)$ and
$g^{(2)}$ is seen as an element of $\mathrm{Aut}(L^+)$). In particular,
$g_d\in N_d$ acts trivially, because $N_d$ stabilizes $V(1)\subset V$.
Now Proposition \ref{fiberoverfp} tells us that the fiber over $U_J$ 
is of the desired form (with $U=U_J$). We have the decomposition 
$U_J=(U_J\cap L^-)\oplus (U_J\cap L^+)$ and 
\[
(gU_J)\cap L^+ = (g^{(2)}U_J)\cap L^+.,\qquad 
\mathrm{pr}^-(gU_J) = \mathrm{pr}^-(g^{(1)}U_J).
\]
This proves the desired claim.
\end{proof}

\begin{cor}
For $U\in X_k$ the preimage $\varphi^{-1}(U)$ is isomorphic to $\bP^{k^2-1}$.
In particular, the projection $\varphi:\G(d,n)\to\Gr(d,n)$ is one-to-one on $X_0\cup X_1$.
\end{cor}
\begin{proof}
Since $U\in X_k$ one has $\dim U\cap L_+ = k = \mathrm{codim} (\mathrm{pr}^- U)$.
Hence $\varphi^{-1}(U)\simeq \bP(\mathrm{Hom}(\bC^k,\bC^k))$.
To prove the "in particular" part we note that  	
$X_0$ is the open dense cell and hence $\varphi$ is a bijection on $X_0$ 
by definition. 
\end{proof}	

\begin{example}
Let $n=4$, $d=2$. Then $X_2$ is a point $\mathrm{span}\{\ell_3,\ell_4\}$ and
the perimage $\varphi^{-1}(X_2)\simeq \bP^3$; outside $X_2$ the projection
$\varphi$ is one-to-one. 
\end{example}

\begin{example}
Let $d=2$ with an arbitrary $n$. Then $\Gr(2,n)=X_0\sqcup X_1\sqcup X_2$ and
$X_2\simeq \Gr(2,n-2)$. The preimage over a point from $X_2$ is isomorphic to 
$\bP^3$ and all other fibers are just points. 
\end{example}

\begin{example}
Assume $d\le n-d$. Then $\G(d,n)=\sqcup_{k=0}^d X_k$ and $X_d=\Gr(d,n-d)$. 
The restriction of $\varphi$ to the deepest component $X_d$ is the fiber bundle
whose fiber over a point $U\in \Gr(d,n-d)$ is equal to $\bP(\mathrm{Hom}(L^-,U))$
(i.e.\ $\bP(L_-^*\T U)$).
\end{example}

\begin{prop}
For any $k=1,\dots,\min(d,n-d)$ the dimension of the preimage $\varphi^{-1}X_k$ is
$d(n-d)-1$.
\end{prop}
\begin{proof}
We note that the codimension of $X_k$ in the Grassmannian $\Gr(d,n)$ is equal to $k^2$.
In fact, 
$$\dim X_k=\dim \Gr(k,n-d) + \dim \Gr(d-k,n-k) = d(n-d) - k^2.$$
Now it suffices to note that for $U\in X_k$ the preimage of $U$ is the projective space of dimension $k^2-1$. 	
\end{proof}

\begin{example}
Let us compute the Poincar\'e polynomial of $\G(3,6)$ as the sum of products of 
Poincar\'e polynomials of $X_k$ and $\bP^{k^2-1}$. One gets $1+2q+4q^2+7q^3+10q^4+11q^5+10q^6+6q^7+3q^8+q^9$. The polynomial is not symmetric and 
hence $\G(3,6)$ is not smooth.
\end{example}

\subsection{Toric degenerations and coordinate rings}
In this subsection we collect the geometric consequences from Theorem \ref{Msum} 
and Corollary \ref{Xcor}.

\begin{cor}
The homogeneous components $H_{m,M}$ of the homogeneous coordinate ring $H$ of $\G(d,n)$ 
are isomorphic to $L^*_{m,M}$ as  $\msl_n^{(d)}$ modules.
\end{cor}
\begin{proof}
Theorem \ref{Msum}  and Corollary \ref{Xcor} imply that modules $L_{m,M}$
are favourable in the sense of \cite{FFL3}. Now the desired property is implied by 
\cite{FFL3}. 
\end{proof}	

Yet another consequence of Theorem \ref{Msum} and Corollary \ref{Xcor} (also via 
\cite{FFL3}) is the following corollary.

\begin{cor}
$\G(d,n)$ admits a flat toric degeneration to the toric varieties with the 
Newton polytopes $X(m,M)$, $m,M>0$. 
\end{cor}

It would be very interesting to compare our description of $H_{m,M}$ and of toric degenerations
with the description of \cite{BSS}.

\appendix
\section{A proof of the Minkowski property (in collaboration with Wojciech Samotij)}
\label{sec:appendix}

Let $P$ be a finite poset. In Section \ref{combinatorics}, we defined the sets 
$S(m,M)\subseteq \bZ_{\ge 0}^P$.
\begin{thm}\label{M+1}
  $S(m,M+1) = S(m,M) + S(0,1)$ for all integers $m, M \ge 0$.
\end{thm}
\begin{proof}
  The inclusion $S(m,M+1) \supseteq S(m,M) + S(0,1)$ is immediate.
  Indeed, if $x \in S(m,M)$ and $y \in S(0,1)$, then, for all $P' \subseteq P$,
  \[
    \sum_{\alpha \in P'} (x_\alpha + y_\alpha) \le m \cdot w(P') + M + 0 \cdot w(P') + 1
  \]
  and thus $x+y = (x_\alpha + y_\alpha)_{\alpha \in P} \in S(m,M+1)$.
  We will prove that the reverse inclusion is a consequence of duality of linear programming \cite{Sch}.

  Given a vector $z \in \mathbb{Z}^P$, define
  \[
    M(z) \coloneqq \max_{P' \subseteq P} \sum_{\alpha \in P'} z_\alpha - m \cdot w(P') \in \mathbb{Z}_{\ge 0}
  \]
  and note that a vector $z \in \mathbb{Z}_{\ge 0}^P$ belongs to $S(m,M)$ if and only if $M(z) \le M$.
  We will prove the following assertion.  We write $e_\delta$ to denote the vector $(\one_{\beta = \delta})_{\beta \in P} \in \mathbb{Z}_{\ge 0}^P$.
  \begin{lem}
    \label{lem:posets-main}
    For every $z \in \mathbb{Z}^P$ such that $M(z) > 0$, there exists some $\delta \in P$ such that $M(z-e_\delta) \le M(z) - 1$.
  \end{lem}

  We first show the lemma yields the inclusion $S(m,M+1) \subseteq S(m,M) + S(0,1)$.
  Fix an arbitrary $z \in S(m,M+1)$.
  We may clearly assume that $z \notin S(m,M)$, as otherwise there is nothing to prove.
  Additionally, we may assume that $z_\alpha \ge 1$ for all $\alpha \in P$, as otherwise we can replace $P$ with its subposet induced by the set $\{\alpha \in P : z_\alpha \ge 1\}$ and remove from $z$ all its zero coordinates.
  Indeed, this operation does not change the value of $M(z)$.
  Since $z \notin S(m,M)$, we have $M(z) = M+1 > 0$ and thus we may invoke Lemma~\ref{lem:posets-main} to find a $\delta \in P$ satisfying the assertion of the lemma.
  Clearly, $e_\delta$ belongs to $S(0,1)$ and $z - e_\delta \in \mathbb{Z}_{\ge 0}^P$ thanks to our assumption that $z \ge 1$.
  Further, $z - e_\delta \in S(m,M)$, as $M(z - e_\delta) \le M(z) - 1 = M$.
  We now turn to the proof of Lemma~\ref{lem:posets-main}.

  Let $\Pc$ be the poset obtained from $P$ by adding to it two new elements $\sigma$ and $\tau$ so that $\sigma \le \alpha \le \tau$ for all $\alpha \in P$.
  Fix an arbitrary vector $z \in \mathbb{Z}^P$ and consider the following integer program (P):
  \begin{align*}
    \textbf{maximise} & \quad \sum_{\substack{\alpha \in \Pc, \beta \in P \\ \alpha < \beta}} z_\beta \cdot f_{\alpha\beta} - m \cdot \sum_{\alpha \in P} f_{\alpha\tau} \\
    \textbf{subject to} & \qquad \sum_{\substack{\alpha \in \Pc \\ \alpha < \beta}} f_{\alpha\beta} = \sum_{\substack{\gamma \in \Pc \\ \beta < \gamma}} f_{\beta\gamma} \le 1 \quad \text{for all $\beta \in P$} \\
    & \qquad f_{\alpha\beta} \in \{0,1\} \quad \text{for all $\alpha, \beta \in \Pc$ with $\alpha < \beta$}.
  \end{align*}

  \begin{claim}
    \label{claim:P-value-Mz}
    The value of (P) is $M(z)$
  \end{claim}  
  \begin{proof}
    Let $P'$ be an arbitrary subposet of $P$ and let $\cC_{P'}$ be a collection of $w(P')$ pairwise-disjoint chains whose union is $P'$.
    Adding $\sigma$ and $\tau$ to each chain in $\cC_{P'}$ yields a collection of $w(P')$ chains of the form $\sigma < \alpha_1 < \dotsb < \alpha_\ell < \tau$, where $\alpha_1, \dotsc, \alpha_\ell \in P'$.
    It is not hard to verify that the vector $(f_{\alpha\beta})_{\alpha < \beta}$ such that  $f_{\sigma\alpha_1} = \dotsb = f_{\alpha_\ell\tau} = 1$ for every chain in $\cC_{P'}$ and $f_{\alpha\beta} = 0$ otherwise is a feasible solution to (P) whose cost is
    \begin{equation}
      \label{eq:cost-P'}
      \sum_{\beta \in P'} z_\beta - m \cdot w(P').
    \end{equation}
    We conclude that the value of (P) is at least the maximum of~\eqref{eq:cost-P'} over all $P' \subseteq P$, that is, $M(z)$.

    Conversely, observe that every feasible solution $f = (f_{\alpha\beta})_{\alpha < \beta}$ to (P) corresponds to a collection $\cC_f$ of $|\{\alpha \in P : f_{\alpha\tau} = 1\}|$ chains of the form $\sigma < \alpha_1 < \dotsb < \alpha_\ell < \tau$, where $\alpha_1, \dotsc, \alpha_\ell \in P$ and $f_{\sigma\alpha_1} = \dotsb = f_{\alpha_\ell\tau} = 1$, such that every $\alpha \in P$ belongs to at most one chain in $\cC_f$.
    Let $P_f' \subseteq P$ be the set of elements of $P$ that appear on some chain in $\cC_f$ and note that the value of the cost function at $f$ is
    \[
      \sum_{\beta \in P_f'} z_\beta - m \cdot |\{\alpha \in P : f_{\alpha\tau} = 1\}| \le \sum_{\beta \in P_f'} z_\beta - m \cdot w(P_f') \le M(z),
    \]
    where the first inequality holds as $P_f'$ is a union of $|\{ \alpha \in P : f_{\alpha\tau} = 1\}|$ chains.
  \end{proof}

  Let us now write (P) in a more compact (but equivalent) form.
  Set $R \coloneqq \{(\alpha,\gamma) \in \Pc^2 : \alpha < \gamma\}$ and define the following two $|P| \times |R|$ matrices:
  \[
    A \coloneqq \bigl(\one_{\beta = \gamma} - \one_{\beta = \alpha}\bigr)_{\beta \in P, (\alpha,\gamma) \in R}
    \qquad
    \text{and}
    \qquad
    A' \coloneqq \bigl(\one_{\beta = \gamma}\bigr)_{\beta \in P, (\alpha,\gamma) \in R}.
  \]
  Writing $z_\tau \coloneqq -m$, we may reformulate (P) as follows:
  \begin{align*}
    \textbf{maximise} & \qquad \sum_{(\alpha,\beta) \in R} f_{\alpha\beta}z_\beta \\
    \textbf{subject to} & \qquad Af = 0, A'f \le 1, f \ge 0, f \in \mathbb{Z}.
  \end{align*}
  (Note that the constraint $f \le 1$ in the original formulation of (P) follows from the constraints $A'f \le 1$, $Af = 0$, and $f \ge 0$.)

  \begin{claim}
    The $2|P| \times |R|$ matrix $\begin{pmatrix}A \\ A'\end{pmatrix}$ is totally unimodular.
  \end{claim}
  \begin{proof}
    Let $B$ be an arbitrary square submatrix of $\begin{pmatrix}A \\ A'\end{pmatrix}$ and let $Q \subseteq P$ and $Q' \subseteq P$ be the rows of $A$ and $A'$, respectively, that appear in $B$.
    Let $B'$ be the matrix obtained from $B$ by the following row operations:
    For every $\beta \in Q \cap Q'$, subtract $A'_\beta$ from the row $A_\beta$.
    Note that every column of $B'$ contains at most two nonzero entries.
    Moreover, if a column has two nonzero entries, then these are $1$ and $-1$.
    Observe that this implies that $|\det B'| \in \{0,1\}$.
    Indeed, if every column of $B'$ has exactly two nonzero entries, then the rows of $B'$ sum to the zero vector.
    Othrewise, expanding $\det B'$ in a column with at most one nonzero entry shows that $|\det B'|$ is either zero or equals $|\det B''|$ for some matrix $B''$ whose columns have the same property as the columns of $B'$.
  \end{proof}

  Duality of linear programming and the well-known fact that total unimodularity of the constraint matrix of a linear program guarantees that its optimal value is attained on some integer vector (see \cite[Chapter~19]{Sch}) imply that the maximum of (P) equals the minimum of the following integer version of its dual (D):
  \begin{align*}
    \textbf{minimise} & \qquad \sum_{\beta \in P} g_\beta \\
    \textbf{subject to} & \qquad h_\gamma - h_\alpha + g_\gamma \ge z_\gamma \text{ for all $(\alpha,\gamma) \in R$} \\
    & \qquad g \in \mathbb{Z}_{\ge 0}^P \text{ and } h \in \mathbb{Z}^P,
  \end{align*}
  where $h_\sigma = h_\tau = g_\tau \coloneqq 0$ and, as before, $z_\tau \coloneqq -m$.

  Suppose that the common value of (P) and (D) is nonzero and let $g \in \mathbb{Z}_{\ge 0}^P$ and $h \in \mathbb{Z}^P$ be two feasible vectors achieving the minimum of (D).
  Choose an arbitrary $\delta \in P$ with $g_\delta \ge 1$.
  It is straightforward to verify that $g - e_\delta$ and $h$ are a feasible solution to (D) with $z$ replaced by $z - e_\delta$.
  Since clearly $\sum_{\beta \in P} (g-e_\delta)_\beta = \sum_{\beta \in P}g_\beta - 1$, the we may use Claim~\ref{claim:P-value-Mz} to conclude that
  \[
    M(z - e_\delta) \le M(z) - 1,
  \]
  as claimed.
\end{proof}

\end{document}